\documentclass[reqno,12pt]{amsart}

\usepackage{mathrsfs}

\newtheorem{thm}{Theorem}[section]
\newtheorem{lm}[thm]{Lemma}

\newtheorem{conj}[thm]{Conjecture}
\newtheorem{pro}[thm]{Proposition}

\newtheorem{ex}[thm]{Example}
\newtheorem{rmk}[thm]{Remark}
\newtheorem{open}[thm]{Open Problem}
\newtheorem{constr}[thm]{}

\newenvironment{pf}
	{\vskip 0.7em\par\noindent
	{\bf Proof.}
	}
	{\qed
	\vskip 0.7em\par
	}
	
\usepackage[usenames]{color}

\usepackage{cite}
\usepackage[cp1250]{inputenc} 

\usepackage{a4wide} 
\usepackage[stable]{footmisc} 
\usepackage{fancyhdr} 
\usepackage{layout} 

\usepackage{amsthm,amsmath,amstext,amsbsy,amsfonts,amssymb} 
\usepackage{stmaryrd} 
\usepackage{turnstile} 
\usepackage{mathtools} 

\usepackage{graphicx} 
\usepackage[all,cmtip]{xy} 
\usepackage{fancybox} 
\usepackage{hhline} 
\usepackage{colortbl} 
\usepackage{color} 
\usepackage{pstricks} 
\usepackage{qtree} 
\usepackage{tikz} 

\usepackage{enumerate} 
\usepackage[ps2pdf]{hyperref} 
\usepackage{index} 
\usepackage{longtable} 
\usepackage{cite} 

\usepackage{ifthen} 
\usepackage{xstring} 
\usepackage{calc} 
\usepackage{forloop} 

\usepackage{xspace} 
\ifx\uplodadedmylogic\undefined
	\def\uplodadedmylogic{}

\newcommand{\str}[1]{\StrLen{#1}[\tmplength]\ifthenelse{1=\tmplength{}}{\mathcal{#1}}{\underline{#1}}}

\newcommand{\A}{\str{A}}
\newcommand{\B}{\str{B}}

\newcommand{\stru}[2]{\lara{#1,#2}} 
\newcommand{\lang}[1]{\lara{#1}} 

\newcommand{\thn}[1]{\mathrm{#1}} 

\newcommand{\PrA}{\thn{Pr}}

\newcommand{\DeLOs}[2]{{\tiny\ifthenelse{\equal{#1}{#2}}{\thn{DeLO}^{#1}}{\ifthenelse{\equal{#2}{}}{\thn{DeLO}^{#1}}{\thn{DeL0}^{\vpair{$#1$}{$#2$}}}}}}

\newcommand{\lekv}{\leftrightarrow}
\newcommand{\limp}{\rightarrow}
\renewcommand{\land}{\,\&\,}

\newcommand{\IFF}{\Leftrightarrow}
\newcommand{\PAK}{\Rightarrow}
\newcommand{\KAP}{\Leftarrow}

\newcommand{\Th}[1]{Th(#1)}

\newcommand{\landef}[2][{}]{\csname Ln#2\endcsname{#1} = \csname L#2\endcsname{#1}}
\newcommand{\landefl}[2][{}]{$\landef[#1]{#2}$, kde \csname Ld#2\endcsname{#1}}

\def\LnfGrpA/{aditivní jazyk teorie grup}
\def\LnFGrpA/{Aditivní jazyk teorie grup}

\def\LnfOrd/{jazyk teorie uspořádání}
\def\LnFOrd/{Jazyk teorie uspořádání}

\def\LnfPrv/{výrokový jazyk s množinou prvovýroků $\Prv$}
\def\LnFPrv/{Výrokový jazyk s množinou prvovýroků $\Prv$}

\def\LnfAr/{jazyk aritmetiky}
\def\LnFAr/{Jazyk aritmetiky}

\def\LnfArZ/{jazyk $\Z$-aritmetiky}
\def\LnFArZ/{Jazyk $\Z$-aritmetiky}

\def\LnfArPA/{jazyk Peanovy aritmetiky}
\def\LnFArPA/{Jazyk Peanovy aritmetiky}

\def\LnfArZPA/{jazyk $\Z$-Peanovy aritmetiky}
\def\LnFArZPA/{Jazyk $\Z$-Peanovy aritmetiky}

\def\LnfArA/{aditivní jazyk aritmetiky}
\def\LnFArA/{Aditivní jazyk aritmetiky}

\def\LnfArAA/{jazyk aditivní aritmetiky}
\def\LnFArAA/{Jazyk aditivní aritmetiky}

\def\LnfArZAA/{jazyk $\Z$-aditivní aritmetiky}
\def\LnFArZAA/{Jazyk $\Z$-aditivní aritmetiky}

\def\LnfArLA/{jazyk lineární aritmetiky}
\def\LnFArLA/{Jazyk lineární aritmetiky}

\def\LnfArZLA/{jazyk $\Z$-lineární aritmetiky}
\def\LnFArZLA/{Jazyk $\Z$-lineární aritmetiky}

\def\LnfArKLA/{jazyk $\kappa$-lineární aritmetiky}
\def\LnFArKLA/{Jazyk $\kappa$-lineární aritmetiky}

\def\LnfArKZLA/{jazyk $\kappa$-$\Z$-lineární aritmetiky}
\def\LnFArKZLA/{Jazyk $\kappa$-$\Z$-lineární aritmetiky}

\def\LnfMod/{jazyk modulů nad okruhem}
\def\LnFMod/{Jazyk modulů nad okruhem}

\def\LnfSuc/{jazyk následníka}
\def\LnFSuc/{Jazyk následníka}

\def\LnfSucO/{jazyk následníka s nulou}
\def\LnFSucO/{Jazyk následníka s nulou}

\newcommand{\Prv}{\mathbb{P}}
\def\DNF/{DNF}
\def\CNF/{CNF}

\fi

\ifx\uplodadedmygeneral\undefined
	\def\uplodadedmygeneral{}
\newcommand{\N}{\mathbb{N}}
\newcommand{\Z}{\mathbb{Z}}

\newcommand{\Primes}{\mathbb{P}}

\newcommand{\lr}[1]{\{#1\}}

\newcommand{\lara}[1]{\langle #1\rangle}

\newcommand{\set}[2]{\lr{#1;#2}}

\newcommand{\sbs}{\subseteq}

\newcommand{\vect}[1]{\overline{#1}}
\newcommand{\restr}{\upharpoonright}
\newcommand{\Restr}[2]{#1\restr #2}

\newcommand{\tuple}[1]{(#1)} 

\newcommand{\vpairg}[3]{{\newbox\horni \newbox\dolni \setbox\horni=\hbox{#1} \setbox\dolni = \hbox{#2} \newdimen\zdeposun \newdimen\wddolni \newdimen\wdhorni \setlength{\wddolni}{\wd\dolni} \setlength{\wdhorni}{\wd\horni} \setlength{\zdeposun}{-\wdhorni-\wddolni/2+\wdhorni/2} {\lower -#3 \hbox{#1}}\kern\zdeposun {\lower #3 \hbox{#2}}}} 
\newcommand{\vpair}[2]{\vpairg{#1}{#2}{.85ex}}

\fi
\ifx\uplodadedmytextstructuring\undefined
	\def\uplodadedmytextstructuring{}
\newcommand{\idf}[3]{\label{#1}\csname index#2\endcsname[|textbf]#3} 

\newcommand{\refer}[1]{(\ref{#1})}
\newcommand{\benum}{\begin{enumerate}}
\newcommand{\eenum}{\end{enumerate}}

\newcommand{\bitem}{\begin{itemize}}
\newcommand{\eitem}{\end{itemize}}

\newcommand{\uvz}[1]{``#1''} 

\newcommand{\axiominfo}[7]{
\begin{tabular}{#7}
Znění: & #1\\
Význam: & #2
\ifthenelse{\equal{#4}{}}{}{\\#3 & #4}%
\ifthenelse{\equal{#6}{}}{}{\\#5 & #6}%
\end{tabular}
}

\newcommand{\repeatstatement}[4][thm]
{
{
\renewcommand{\thestatement}{\ref{#2}}
\begin{#1}[#3]
#4
\end{#1}
\addtocounter{statement}{-1}
}
}

\newcounter{thislevel}
\newcommand{\mysect}[2]{\setcounter{thislevel}{\value{mysectlevel}} \addtocounter{thislevel}{#1}
\if\arabic{thislevel}0\part{#2}\fi
\if\arabic{thislevel}1\chapter{#2}\fi
\if\arabic{thislevel}2\section{#2}\fi
\if\arabic{thislevel}3\subsection{#2}\fi
\if\arabic{thislevel}4\subsubsection{#2}\fi
\if\arabic{thislevel}5\paragraph{#2}\fi
\if\arabic{thislevel}6\subparagraph{#2}\fi}
\fi

\newcommand{\veps}[3]{\varepsilon(#1)_{#2#3}}
\newcommand{\vep}[1]{\veps{#1}{}{}}
\newcommand{\vepe}{\varepsilon}

\newcommand{\Mgood}[1]{M^{good}_{#1}(\B)}
\newcommand{\Pmat}{\Mgood{\Primes}}

\newcommand{\bprime}{P}

\newcommand{\dotminus}{\stackrel{\textstyle.}{\raisebox{0pt}[0.4\height]{$-$}}}

\newcommand{\rad}[1]{\mathrm{rad}(#1)}

\newcommand{\ISone}{\mathrm{I}\Sigma_1}

\begin{document}

\title{Fermat's Last Theorem and Catalan's Conjecture in Weak Exponential Arithmetics}
\author{\textsc{Petr~Glivick\' y}}
\address{\textsc{Petr~Glivick\' y}: University of Economics, Department of Mathematics\\ 
Ekonomick\' a 957, 148 00 Praha~4, Czech Republic} 
\address{\textsc{\hphantom{Petr~Glivick\' y}}: Academy of Sciences of the Czech Republic, Institute of Mathematics\\ 
\v Zitn\' a 25, 115 67 Praha~1, Czech Republic}
\email{petrglivicky@gmail.com}
\thanks{The research leading to these results has received funding from the European Research Council under the European Union's Seventh Framework Programme (FP7/2007-2013) / ERC grant agreement n${}^{\circ}$\ 339691. This paper was processed with contribution of long term institutional support of research activities by Faculty of Informatics and Statistics, University of Economics, Prague.}

\author{\textsc{V\' it\v ezslav~Kala}}
\address{\textsc{V\' it\v ezslav~Kala}: Charles University,
Faculty of Mathematics and Physics, Department of Algebra, Sokolovsk\' a 83,
186 00 Praha 8, Czech Republic}
\address{\textsc{\hphantom{V\' it\v ezslav~Kala}}: Max-Planck-Institut f\" ur Mathematik, Vivatsgasse 7, D-53111 Bonn, Germany}
\email{vita.kala@gmail.com}

\subjclass[2010]{03F30 (Primary), 11U10, 03H15, 03C62 (Secondary)}

\date{February 11, 2016}

\keywords{weak arithmetics, Fermat's Last Theorem, Catalan's Conjecture}

\begin{abstract}
We study Fermat's Last Theorem and Catalan's conjecture in the context of weak arithmetics with exponentiation.
We deal with expansions $\stru{\B}{e}$ of models of arithmetical theories (in the language $L=\lang{0,1,+,\cdot,\leq}$) by a binary (partial or total) function $e$ intended as an exponential. We provide a general construction of such expansions 
and prove that it is universal for the class of all exponentials $e$ which satisfy a certain natural set of axioms $Exp$.

We construct a model $\stru{\B}{e}\models \Th{\N} + Exp$ and a substructure $\stru{\A}{e}$ with $e$ total and $\A\models\PrA$ (Presburger arithmetic) 
such that in both $\stru{\B}{e}$ and $\stru{\A}{e}$ 
Fermat's Last Theorem for $e$ is violated by cofinally many exponents $n$ and (in all coordinates) cofinally many pairwise linearly independent triples $a,b,c$.

On the other hand, under the assumption of ABC conjecture (in the standard model), we show that Catalan's conjecture for $e$ is provable in $\Th{\N}+Exp$ (even in a weaker theory
) and thus holds in $\stru{\B}{e}$ and $\stru{\A}{e}$. 

Finally, we also show that Fermat's Last Theorem for $e$ is provable (again, under the assumption of ABC in $\N$) in $\Th{\N}+Exp+$\uvz{coprimality for $e$}.
\end{abstract}

\maketitle

\section{Introduction}
Wiles's proof of Fermat's Last Theorem (FLT) \cite{Wil95} has stimulated a lively discussion on how much is actually needed for the proof. 
Despite the fact that the original proof uses set-theoretical assumptions unprovable in Zermelo-Fraenkel set theory with axiom of choice (ZFC) (namely, the existence of Grothendieck universes; see \cite{McLar10} for more on this topic), it is widely believed that 
\begin{quote}
{\em\uvz{certainly much less than ZFC is used in principle, probably nothing beyond PA, and perhaps much less than that.}} \cite[p. 359]{McLar10}
\end{quote} 
McLarty showed that Grothendieck's apparatus can be formalized in finite order arithmetic (hence in ZFC) \cite{McLar11} and partially even in second order arithmetic \cite{McLar12}.
Macintyre \cite[Appendix]{Mac11} proposed and sketched a project of formalizing Wiles's proof in Peano arithmetic.

Friedman even conjectured that Fermat's Last Theorem\footnote{Friedman actually made a much stronger conjecture concerning {\em\uvz{every theorem published in the Annals of Mathematics whose statement involves only finitary mathematical objects (i.e., what logicians call an arithmetical statement)}} in place of FLT.}
is provable in the so called elementary function arithmetic (EFA) \cite{Fri99}. Here, EFA is a theory in the language $\lang{0,1,+,\cdot,\exp,\leq}$ which extends the usual quantifier free axioms for $0,1,+,\cdot,\exp,\leq$ by the scheme of bounded induction (see \cite[section 2, theory EA]{Avi03} for a possible axiomatization)\footnote{Let us note that (up to a change of language) EFA is equivalent to $\mathrm{I}\Sigma_0(exp)$ or $\mathrm{I}\Sigma_0$ + \uvz{$2^x$ is total} (see \cite[I.1.28]{HParitmeticbook} and the discussion after Proposition V.1.3 ibid.)}.

Some nice results in the direction of 
these conjectures are due to Smith, who in \cite{Smi92} proved that the theory $\mathrm{IE}_1$ of bounded existential arithmetic (a $\lang{0,1,+,\cdot,\leq}$-theory containing induction only for bounded existential quantifications of open formulas, hence even weaker then Friedman's EFA) proves Fermat's Last Theorem for some small even exponents $n$ (e.g., for $n = 4,\ 6,\ 10$). In the same paper, Smith also proves some special cases of FLT in the even weaker theory IOpen + \uvz{every two elements have a greatest common divisor}. Here, IOpen is the extension of Robinson arithmetic by induction for all quantifier-free formulas.

These results, however, can not be strengthened down to IOpen.
In a rather well-known paper \cite{She64}, Shepherdson constructed a (recursive) model of IOpen where the equation $x^3+y^3=z^3$ has a non-zero solution.
Recently, Ko\l odziejczyk in \cite{Kol11} extended Shepherdson's method to Buss's arithmetic $T^0_2$ (containing induction for sharply bounded formulas in Buss's language; see for example \cite[V.4.4]{HParitmeticbook}). In particular, he showed that $T^0_2$ does not prove Fermat's Last Theorem for $n=3$.

\medskip

The goal of this paper is to study Fermat's Last Theorem and Catalan's Conjecture in the context of weak arithmetics, including theories where exponentiation is not definable from addition and multiplication.
For this purpose we consider structures and theories in the language $L^e=\lang{0,1,+,\cdot,e,\leq}$, where the symbol $e$ is intended for a (partial or total) binary exponential.
 
In these structures, in general, no induction for multiplication or exponentiation is assumed. However, we are able to construct interesting $L^e$-structures whose $\lang{0,1,+,\cdot,\leq}$-fragments are models of Peano arithmetic or even of the complete true arithmetic $\Th{\N}$. In particular, this means that in these models we then work with two different exponential functions, the definable ``strong" one (always denoted as $x^y$) and the exponential $e$ from the language (denoted $e(x, y)$) which is usually ``weak'' (not satisfying induction). The reader should keep in mind that in such cases the role of the definable exponential $x^y$ is purely auxiliary (in the proofs) while theorems are always stated with the exponential $e$.

\medskip

We show that Fermat's Last Theorem for $e$  (i.e., the statement \uvz{$e(a,n)+e(b,n)=e(c,n)$ has no non-zero solution for $n>2$}) is not provable in the $L^e$-theory $\Th{\N}+ Exp$, where $\Th{\N}$ stands for the complete theory of the structure $\N=\stru{\N}{0,1,+,\cdot,\leq}$ and $Exp$ is a natural set of axioms for $e$ (consisting mostly of elementary identities; see Section {\ref{sect:violationFLT}}).

In more detail -- we construct a model $\stru{\B}{e}\models \Th{\N} + Exp$ and a substructure $\stru{\A}{e}$ with $e$ total and $\A\models\PrA$ (Presburger arithmetic) such that in both $\stru{\B}{e}$ and $\stru{\A}{e}$ Fermat's Last Theorem for $e$ is violated by cofinally many exponents $n$ and cofinally (in all coordinates) many pairwise linearly independent triples $a,b,c$. 
Moreover, we show that for any fixed $y$ the function $e(x,y)$ is a definable (in $\B$) function of $x$, and that $e$ is definable in the expansion $\stru{\B}{\mathcal{N}}$ of $\B$ by a predicate $\mathcal{N}(x)$ expressing \uvz{$x$ is a standard number}.
The results are summarized in Theorem \ref{thm:FLT}. 

On the other hand, under the assumption of ABC conjecture%
\footnote{Mochizuki recently announced a proof of ABC conjecture \cite{Moc12}. However, its correctness has not yet been completely verified.}
(in the standard model), we show that Catalan's conjecture for $e$ (\uvz{the only solution of $e(a,n)-e(b,m)=1$ with $a,b,m,n> 1$ is $a=m=3$, $b=n=2$}) is provable in $\Th{\N}+Exp$ (even in a weaker theory -- see Section \ref{sect:Catalan} and Theorem \ref{thm:Catalan}) and thus holds in $\stru{\B}{e}$ and $\stru{\A}{e}$. (Of course, we also have to use Catalan's conjecture in the standard model, as proved by Mih\u{a}ilescu \cite{Mih04}.)

This gives an interesting separation of the strengths of the two famous Diophantine problems.

As we note in Section \ref{sect:coprime}, a crucial property for the validity of Fermat's Last Theorem is the \uvz{coprimality} of $e$, i.e., the 
statement that if $x$ and $y$ are coprime, then so are $e(x, a)$ and $e(y, b)$. Assuming this, we can again use the ABC conjecture and show in 
Theorem \ref{thm 6.1} that Fermat's Last Theorem holds for exponentials $e$ which satisfy this coprimality condition.

\medskip

Nevertheless, we do not know whether there is a model $\stru{\B}{e}\models \Th{\N} + Exp$ (or at least $\stru{\B}{e}\models \mathrm{IOpen} + Exp$) with $e$ total, where Fermat's Last Theorem for $e$ does not hold (see Open Problem \ref{open:totality}). Note that our model $\A$ (which carries a total $e$ violating FLT) satisfies $\PrA + \mathrm{OpenTh}(\N)$, where $\mathrm{OpenTh}(\N)$ stands for the set of all open formulas true in the the standard model $\N$.

\medskip

The results above are obtained using a general construction of 
an exponential $e$ over a background model $\B$ of a sufficiently strong (at least $\ISone$) arithmetical theory (in the language $L=\lang{0,1,+,\cdot,\leq}$). We describe it in Section {\ref{sect:generalconstruction}} and prove that the construction is universal for the class of all exponentials $e$ which satisfy the set of axioms (e1)-(e7) from $Exp$ (see Proposition \ref{pro:exp}). 

\medskip

To relate the presented results with Shepherdson's, let us note that Shepherdson's model is a structure in the language $L=\lang{0,1,+,\cdot,\leq}$, which is too weak even to define any good notion of exponential.

On the other hand, in our model, its $L$-part $\B$ is a model of $\Th{\N}$ and thus \uvz{well-behaved}. The \uvz{weakness} of the model comes from the properties of the exponential $e$ (which differs dramatically from the exponential $x^y$ definable in $\B$). In particular, since both exponentials -- $x^y$ and $e$ -- coincide for all standard exponents $n$, in our model FLT for $e$ holds for all exponents $n\in\N$.

\medskip

Let us also mention that Ml\v cek \cite{Mlc76} has studied the existence of unboundedly many twin primes (i.e., pairs of primes $p$ and $p+2$) in models of weak arithmetics. He constructed three models of the theory $\PrA + \mathrm{OpenTh}(\N)$ (the same theory which is satisfied by our model $\A$) with only boundedly many primes, with unboundedly many primes but boundedly many twin primes, and with unboundedly many twin primes, respectively.
Let us note that if $\mathcal B$ is a model of Peano arithmetic, then our model $\mathcal A$ contains unboundedly many primes 
 -- see the remark immediately following the proof of Lemma \ref{pres}.
It would be interesting to consider the question of twin primes together with Fermat's Last Theorem in more detail.

\medskip

One can of course consider extending the above methods to the study of other exponential Diophantine equations. They can certainly be used
to construct solutions to various homogeneous equations (such as was $a^n+b^n=c^n$) -- the crucial thing is having an analogue of Lemma \ref{balog}.
One can also proceed similarly if the given equation can be made homogeneous by a suitable substitution for the variables.
However, note for example that if Fermat's Last Theorem for $n=3$ holds in $\mathcal B$, the equation $a^{3n}+b^{3n}=c^{3n}$ can have no solutions even with the
new exponential.

In the case of non-homogeneous equations one can sometimes expect to be able to use the ABC Conjecture, as we illustrated by the case of Catalan's Conjecture $a^n-b^m=1$. See the remarks in Section \ref{sect:coprime} concerning the importance of coprimality for the exponential. In any case, obtaining a truly general theorem seems to be a hard and interesting question for further research.

\section*{Acknowledgments}

We want to thank Professors Leonard Lipshitz and Josef Ml\v cek for their advice and suggestions which have helped improve the quality of this paper.
We also thank the anonymous referees for a careful reading of the paper and their useful comments.

\section{Preliminaries}

In this paper, $\N$ denotes the set $\N=\lr{0,1,\ldots}$ of natural numbers.

We shall denote models of theories by ``calligraphic'' letters $\mathcal M$ and their underlying sets by normal letters $M$.

\begin{constr}
{\bf Arithmetical theories}
{\rm By the language of arithmetic we mean the language $L=\lang{0,1,+,\cdot,\leq}$. 
The $L$-theory $\ISone$ is the extension of Robinson arithmetic by the scheme of induction for all $\Sigma_1$-formulas, i.e., for formulas of the form $(\exists x_0,\ldots,x_{n-1})\psi(\vect{x},\vect{y})$, where $\psi$ contains only bounded quantifiers. Let us note that the usual (G\"odel's) coding of formally finite sets is available in $\ISone$. We are going to use the coding at many places of the following text, mostly without explicitly mentioning it.

Let $\B$ be a model of $\ISone$. 
We say that a set $X\sbs B^n$ is coded in $\B$ (or, equivalently, finite in the sense of $\B$) if there is an element $s\in B$ with $\B\models$ \uvz{$s$ is a set} and for any $u\in B^n$ one has $\B\models u\in s$ if and only if $u\in X$. 
Let us note that any bounded part of a $\Sigma_1$-definable set in $\B$ is coded in $\B$. However, when dealing with sets definable in an extended language (such as $L^e$), this no longer needs to be true.

Also note that the usual exponential $x^y$ is $\Delta_1$-definable in $\ISone$. Further on, we will strictly use the notation $x^y$ for the definable exponential, while keeping different notation for other \uvz{exponentials} with which we will work.

\medskip

Presburger arithmetic $\PrA$ is the complete theory $\Th{\stru{\N}{0,1,+,\leq}}$ of the additive structure of natural numbers. It is well-known that $\PrA$ is equivalent to the theory with the following axioms: 

\benum[(Pr1)]
\item $0\neq z+1$,
\item $x\neq 0 \limp (\exists z)(x=z+1)$,
\item $x+z=y+z \limp x=y$,
\item $x+0=x$,
\item $x+(y+z)=(x+y)+z$,
\item $x+y=y+x$,
\item $x\leq y \lekv (\exists z)(x+z=y)$,
\item $(\exists y)(ny\leq x<n(y+1))$, for all $0<n\in\N$.
\eenum

(Note that (Pr8) is equivalent to the induction scheme for all formulas in the language $\lang{0,1,+,\leq}$.)

For an $L$-structure $\A$, by writing $\A\models\PrA$ we mean that all the axioms above are true in $\A$ (this is, of course, a harmless abuse of notation).}
\end{constr}

\begin{constr}
{\bf Good matrices} \label{goodmat}
{\rm
In the following sections, we will often need to work with infinite (even in the sense of $\B$) matrices of elements from our background model $\B\models \ISone$, i.e., with matrices of the form $M=(M_{ij})_{i,j\in B}$, with $M_{ij}\in B$.

Unlike the addition, multiplication of such matrices can not be generally defined. In fact, there are two obstacles in defining the product $P=MN$ of matrices $M$ and $N$:
\bitem
\item We want to have $P_{ij}=\sum_{k\in B}M_{ik}N_{kj}$.
However, this sum may add up to infinity when both the $i$-th row of $M$ and the $j$-th column of $N$ are allowed to contain unboundedly many non-zero elements.
\item Even a bounded sum $\sum_{k<b}a_k$ may not exist in $B$ if $(a_k)_{k<b}$ is not coded in $\B$.
\eitem

If in both $M,N$, each row and each column contain only boundedly many non-zero elements and all these bounded initial segments are coded in $\B$, the product $MN$ is correctly defined but it may contain a column (or row) with unboundedly many non-zero elements. In order to prevent this, we need that, in $M$, any bounded set of columns has a common upper bound for the number of rows containing non-zero elements in these columns.

\medskip 

That is why we define a matrix $M=(M_{ij})_{i,j\in B}$ to be good in $\B$ if the following hold:\\
For any $J\in B$ there is $I=I_M(J)\in B$ such that
\benum[i)]
\item all non-zero values $M_{ij}$ from the first $J$ columns are in the first $I$ rows,
\item the restricted matrix $(M_{ij})_{i<I, j<J}$ is coded in $\B$.
\eenum

Note that the above condition may be equivalently formulated as follows: $M$ is good in $\B$ if and only if for any $J\in B$ the set $\set{\tuple{i,j,M_{ij}}}{j<J,\ i\in B,\ M_{ij}\neq 0}$ is coded in $\B$.

\medskip

We will further assume that all matrices are of the form $(M_{ij})_{i,j\in B}$ (i.e. $B\times B$-matrices) and identify any $X\times Y$-matrix $(N_{ij})_{i\in X, j\in Y}$, where $X,Y\sbs B$ (not necessarily infinite), with the matrix $(M_{ij})_{i,j\in B}$, where $M_{ij}=N_{ij}$ for $(i,j)\in X\times Y$ and $M_{ij}=0$ otherwise.

We denote by $\Mgood{B}$ the set of all $B\times B$-matrices over $B$ good in $\B$.

\begin{lm}
\label{goodmultip}
$\Mgood{B}$ is closed under matrix multiplication.
\end{lm}
\begin{proof}
Let $M,N\in\Mgood{B}$ and let $I_M$, $I_N$ be some functions witnessing that $M$, $N$, respectively, are good in $\B$. Denote $P=MN$.
Take $J\in B$ and $j<J$. We have $P_{ij}=\sum_{k<I_N(J)}M_{ik}N_{kj}$. For $i>I=I_M(I_N(J))$, we get $P_{ij}=\sum_{k<I_N(J)}0\cdot N_{kj}=0$. Also clearly $(P_{ij})_{i<I, j<J}$ is coded in $\B$, since $(M_{ik})_{i<I,k<I_N(J)}$ and $(N_{kj})_{k<I_N(J), j<J}$ are.
\end{proof}
}
\end{constr}

\begin{constr}
{\bf Semirings} \label{semirings}
{\rm
By a semiring we shall mean a structure $\stru{S}{0,1,+,\cdot}$ equipped with two constants $0,1$ and two associative binary operations $+$ and $\cdot$ such that $+$ is commutative, $x + 0=0 + x=x$, $x\cdot 1=1\cdot x=x$,
$x\cdot(y+z)=x\cdot y+x\cdot z$ and $(x+y)\cdot z=x\cdot z+y\cdot z$. 
 
For semirings $S$ and $T$, a semiring homomorphism $\varphi: S\rightarrow T$ is a map
such that $\varphi(x+y)=\varphi(x)+\varphi(y)$, $\varphi(x\cdot y)=\varphi(x)\cdot\varphi(y)$, $\varphi(0)=0$ and $\varphi(1)=1$.
}
\end{constr}

\section{General construction}
\label{sect:generalconstruction}

In this Section, let $\B\models\ISone$ be a fixed \uvz{background model} and $\A$ a substructure of $\B$. 
We show a method of construction of a function $e: B\times A \rightarrow B$ such that the following properties hold in $\stru{\B}{e}$:

\begin{enumerate}[(e1)]
\item $(x=1 \lor y=0) \lekv e(x,y)=1$, 
\item $x\neq 0 \limp e(x,y)\neq 0$,
\item $e(x,1)=x$, 
\item $e(x,y+z)=e(x,y)\cdot e(x,z)$,
\item $e(\prod_{i<l} x_i,y)=\prod_{i<l} e(x_i,y)$ (right hand side is correct thanks to (e\ref{ecoding})),
\item $e(e(x,y),z) = e(x,yz)$,
\item \label{ecoding} \uvz{for any $b\in B$, the set $\set{\tuple{x,e(x,y)}}{x<b}$ is coded in $\B$},
\end{enumerate}
whenever $y,z\in A$, $x\in B$ and $(x_i)_{i<l}$ is a sequence coded in $\B$ of length $l\in B$.


For the convenience of the reader regarding our definitions, let us note that at the beginning of Section \ref{sect:violationFLT} we introduce an additional axiom (e0) and denote by $Exp$ the axioms (e0) -- (e7). $Exp^\prime$ denotes only (e0) --  (e4). At the beginning of Section \ref{sect:coprime} we introduce (e8). In Section \ref{sect:coprime} we shall also use the following weakening of (e5):

\noindent (e$5^\prime$) $e(xy, z)=e(x, z)\cdot e(y,z)$.

\bigskip

Before delving into the technical construction of the exponential, let us outline the idea. Suppose that we have an exponential $e: B\times A\rightarrow B$ 
satisfying the axioms above. How can we characterize $e$? First of all, by multiplicativity (e5), the values $e(q, y)$ at primes $q$ determine $e$. 
If we write $e(q, y)=\prod_{p\in\mathbb P} p^{\varepsilon(y)_{pq}}$ (where $\Primes$ is the set of prime numbers of $\B$), we can form a matrix
$\vep{y}=(\veps{y}{p}{q})_{p,q\in\Primes}$. Property (e\ref{ecoding}) ensures that the matrix $\vep{y}$ is good in $\B$ (see \ref{goodmat}) and (e4) and (e6) imply $\varepsilon(y+z)=\varepsilon(y)+\varepsilon(z)$ and $\varepsilon(yz)=\varepsilon(y)\varepsilon(z)$  (see the proof of Proposition \ref{pro:exp} for details). 
Thus
\begin{eqnarray}
\nonumber \varepsilon=\vepe^e: A &\rightarrow& \Pmat \\
\nonumber y & \mapsto &(\varepsilon(y)_{pq})_{p,q\in\Primes}
\end{eqnarray}
is a semiring homomorphism, where $\Pmat$ denotes the semiring of $\Primes\times\Primes$-matrices $M$ over $\B$ which are good in $\B$.
Notice also that
\begin{equation}
\varepsilon(y)_{pq} = v_p(e(q,y)) \label{def:vep}
\end{equation}
where $v_p(x)$ is the usual additive $p$-adic valuation of $x$ (in $\B$).

\medskip

Conversely, to construct an exponential $e: B\times A \rightarrow B$, we choose
a homomorphism of semirings 
\begin{eqnarray*}
\vepe: A & \rightarrow & \Pmat \\
y &\mapsto & \vep{y}=(\veps{y}{p}{q})_{p,q\in\Primes}.
\end{eqnarray*}
We denote $v:x\mapsto v(x) = (v_p(x))_{p\in \Primes}$.
The exponential $e=e^\vepe:B\times A \rightarrow B$ is then defined as follows:

\begin{eqnarray}
\nonumber e(0,0) &= & 1, \\
e(0,z) &= & 0, \label{def:e} \\
\nonumber e(x,y) &=& v^{-1}(\vep{y}v(x)),
\end{eqnarray}

for all $0\neq x\in B$, $y,z\in A$, $z\neq 0$, where $\vep{y}v(x)$ denotes the product of matrices, calculated inside $\B$. The product makes sense since both $\vep{y}$, $v(x)$ are good matrices in $\B$. Also, by Lemma \ref{goodmultip}, the vector $\vep{y}v(x)$ is good, i.e., its non-zero part is coded in $\B$. Therefore $v^{-1}(\vep{y}v(x))$ exists in $B$ (note that $v^{-1}((a_p))=\prod p^{a_p}$ for a vector $(a_p)_{p\in\mathbb P}$).

In fact, there is a bijection between these semiring homomorphisms and exponentials:

\begin{pro}
\label{pro:exp}
Let $\B\models \ISone$ and $\A\sbs\B$. Then the maps $e\mapsto \vepe^e$ and $\vepe \mapsto e^\vepe$ defined by \refer{def:vep} and \refer{def:e}, respectively, are mutual inverses and
the following are equivalent:
\bitem
\item The exponential $e=e^{\vepe}: B\times A \rightarrow B$ satisfies (e1) -- (e7).
\item The map $\vepe=\vepe^e: A  \rightarrow  \Pmat$ is a semiring homomorphism.
\eitem

Moreover:
\benum[\ \ \ \ \ \ a)]
\item The exponential $e$ is definable in $\B$ from $\vepe$ and vice versa.
\item For a fixed $y\in A$, the map $x\mapsto e(x,y)$ is definable in $\B$ from $\vep{y}$ and vice versa.
\eenum
\end{pro}

\begin{pf}
It is easy to verify the following:
\bitem
\item (e1) $\IFF$ $\vep{0}=0$
\item (e2) holds by \refer{def:e}, and ensures the correctness of \refer{def:vep}
\item (e3) $\IFF$ $\vep{1}=I$
\item (e4) $\IFF$ $\vep{y+z}=\vep{y}+\vep{z}$
\item (e5) holds by \refer{def:e}, its correctness follows from (e7)
\item (e5) + (e6) $\PAK$ $\vep{y\cdot z}=\vep{y}\cdot\vep{z}$:\\
			\textit{Proof:} $\veps{y\cdot z}{p}{q}=v_p(e(q,yz))=v_p(e(e(q,z),y))=v_p(e(\prod_{r\in\Primes} r^{v_r(e(q,z))},y))=$\\
			$= v_p(\prod_{r\in\Primes}e(r,y)^{v_r(e(q,z))}) = \sum_{r\in\Primes} v_p(e(r,y))\cdot v_r(e(q,z)) = (\vep{y}\cdot \vep{z})_{pq}$.
						
\item (e6) $\KAP$ $\vep{y\cdot z}=\vep{y}\cdot\vep{z}$
\item (e7) $\IFF$ $\vep{y}$ is a good matrix for all $y$:\\
      \textit{Proof:} \uvz{$\PAK$}: Let $y\in A$, $J\in B$ be given, we find $I\in B$ such that conditions i) and ii) from the definition of a good matrix hold for $\vep{y}$.
											It is enough to take $I=1+\max\set{p\in\Primes}{v_p(e(q,y))\neq 0$ for some $J> q\in \Primes}$. This is a correct definition in $\B$	since the sequence $(e(q,y))_{q\in\Primes}$ is coded in $\B$ by (e7).\\
											\uvz{$\KAP$}: Let $y\in A$, $b\in B$ is given. Set $J=b$ and take $I\in B$ such that conditions i) and ii) from the definition of a good matrix hold for $\vep{y}$. 
											Then, in $\B$, we may define the sequence $(e(x,y))_{x<b}$ by the definition \refer{def:e} where we use $v'(x)=(v_p(x))_{p<J}$ instead of $v(x)$ and $\varepsilon'(y)=(\veps{y}{p}{q})_{p<I,q<J}$ instead of $\vep{y}$ (both $(v'(x))_{x<b}$ and $\varepsilon'(y)$ are coded in $\B$ by the assumption).
\item $e^{\vepe^{e}}=e$:\\
			\textit{Proof:} The case $x=0$ is trivial. Suppose $x\neq 0$.\\
			Then $e^{\vepe^{e}}(x,y)=v^{-1}((v_p(e(q,y)))_{pq}\cdot v(x)) = v^{-1}((\sum_{q\in\Primes} v_p(e(q,y))\cdot v_q(x))_p) = $\\
			$= \prod_{p\in\Primes} p^{\sum_{q\in\Primes} v_p(e(q,y))\cdot v_q(x)} = \prod_{q\in\Primes} (\prod_{p\in\Primes} p^{v_p(e(q,y))})^{v_q(x)}=\prod_{q\in\Primes} e(q,y)^{v_q(x)} = e(x,y)$,
			where we use (e5) in the last equality.
\item $\vepe^{e^{\vepe}}=\vepe$:\\
			\textit{Proof:} $\vepe^{e^{\vepe}}(y)_{pq}=v_p(v^{-1}(\vep{y}\cdot v(q)))=(\vep{y}\cdot v(q))_p=\veps{y}{p}{q}$.
\eitem

Now, the main statement follows immediately. The ``moreover'' part is easy.
\end{pf}

\begin{ex}
\label{ex:exp}
\benum[a)]
\item Let $\A=\B$ and $\vep{y}=y I$, for $y\in B$, where $I$ is the identity matrix. Then  $e(x,y)=x^y$ (the original exponential in $\B$).
\item Let $\A=\B$, $f$ an automorphism of $\B$ and $\vep y=f(y)I$, for $y\in B$. Then $e(x,y)=x^{f(y)}$.
\eenum
\end{ex}

In general, not much can be said about possible homomorphisms $\vepe: A\rightarrow \Pmat$ and corresponding exponentials. However, in all important examples of exponentials considered in this paper we will have $\A\models\PrA$. Then more can be said:

\begin{rmk}
\label{rem:zerosmatrices}
If $A$ is closed under subtraction (i.e. under $a-b$ with $a\geq b$), then for $y\neq y'$ the matrices $\vep{y}$, $\vep{y'}$ have different values everywhere on the diagonal. (Otherwise, let say $y'<y$. Then the matrix $\vep{y-y'-1}$ would contain a value $-1\notin B$.) In particular, for such $\A$, the homomorphism $\vepe$ is always injective.

We then get some additional nice properties for the exponential $e$, such as $y>0\limp x|e(x,y)$.
\end{rmk}

Note that we do not know of any example of homomorphism $\vepe:B \rightarrow \Pmat$ with non-diagonal matrices in its range. In Section \ref{sect:coprime} we show that only such homomorphisms can yield an exponential (total on $B$) violating FLT (see also the related Open Problem \ref{open:totality}). 
Nevertheless, the following construction provides us with a possibility to find interesting examples of \uvz{non-diagonal} homomorphisms $\vepe: A \rightarrow \Pmat$, if $\A\sbs\B$ is a suitable substructure.

\begin{constr}\label{construction}
{\bf Construction of $\vepe$}
\end{constr}

For a semiring homomorphism $\vepe: A \rightarrow \Pmat$, the values $\vep{y+n}$, for $n\in\Z$, are uniquely determined by the value $\vep{y}$. 
Therefore we may construct $\vepe$ in the following way:

For $y,z\in A$, we define $y\sim z$ if $|y-z|\in\N$.

We choose:
\begin{itemize}
\item a $\lang{0,+,\dotminus,\cdot}$-substructure $\mathcal{O}$ of $\A$ (where $a\dotminus b = a-b$ if $a\geq b$ and $0$ otherwise) such that every $\sim$-factor $[y]_{\sim}$ of $A$ contains a single element $O_y\in O$ (then $O_0 = 0$, $O_y+O_z = O_{y+z}$ and $O_y\cdot O_z = O_{O_y\cdot O_z}$).
\item a $\lang{0,+,\dotminus,\cdot}$-homomorphism 
\begin{eqnarray*}
\vepe: \mathcal{O} & \rightarrow & \Pmat \\
O_y &\mapsto & \vep{O_y}=(\veps{O_y}{p}{q})_{p,q\in\Primes},
\end{eqnarray*}
such that all elements $\veps{O_y}{p}{p}$ with $O_y\neq 0$ are nonstandard.
\end{itemize}

We sometimes call the elements $Q$ of $\mathcal{O}$ \uvz{zeroes}, as around each of them we have the component $\set{Q+z}{z\in\mathbb Z}$.

\begin{rmk}
\label{Onotalways}
It is not always possible to choose a substructure $\mathcal{O}$ as above. In fact, it is easy to see that for $\A\models\PrA$ such a substructure exists if and only if every $\sim$-factor of $A$ contains an element divisible by all $0<n\in\N$.
\end{rmk}

We may then define

\begin{equation}
\label{def:epsilon}
\vep{y} = \vep{O_y}+\delta_y I,
\end{equation}

for $y\in A$, where $I$ is the identity matrix and $\delta_y = y - O_y$.

\begin{lm}
\label{lm:vep}
Let $\vepe: A \rightarrow \Pmat$ be defined by \refer{def:epsilon}. Then it is a semiring homomorphism.
\end{lm}

\begin{pf}
Clearly, $\vep{0}=0$ and $\vep{1}=I$.

It is $\vep{y+z}= \vep{O_{y+z}}+\delta_{y+z} I = \vep{O_y + O_z} + (\delta_y + \delta_z) I = (\vep{O_y} + \delta_y I) + (\vep{O_z} + \delta_z I) =  \vep{y}+\vep{z}$.

Finally, $\vep{yz} = \vep{(O_y+\delta_y)(O_z+\delta_z)} = \vep{O_y O_z} + \vep{\delta_y O_z} + \vep{\delta_z O_y} + \vep{\delta_y \delta_z} = \vep{O_y}\cdot \vep{O_z} + \delta_y \vep{O_z} + \delta_z \vep{O_y} + \delta_y\delta_z I = (\vep{O_y}+\delta_y I)\cdot(\vep{O_z} + \delta_z I) = \vep{y}\cdot\vep{z}$. (For the sake of clarity, we harmlessly abused the notation a bit by writing $\vep{\delta_y O_y}$ even for $\delta_y<0$.)
\end{pf}

\begin{rmk}
The construction may be further generalized by changing the definition of the equivalence $\sim$. We may define $y\sim z$ if $|y-z|\in D$ where $\mathcal D$ is an initial segment of $\B$ and a substructure of $\A$.
Then the notion of homomorphism has to be modified to \uvz{preserve $\mathcal D$}.
\end{rmk}

\section{Violation of FLT}
\label{sect:violationFLT}

We show that Fermat's Last Theorem for $e$:
%
%
\uvz{$e(a,n)+e(b,n)=e(c,n)$ has no non-zero solution for $n>2$},
%
%
is not provable in the $L^e$-theory $\Th{\N}+ Exp$, where $L^e=\lang{0,1,+,\cdot,e,\leq}$ and $Exp$ consists of the following axioms:

\medskip

\begin{tabular}{ll}
(e0) & \uvz{$e:B\times A \rightarrow B$ for some substructure $\A$ of $\B$ with $\A\models\PrA$},\\
\end{tabular}\\
\indent\ \ axioms (e1) -- (e7) from Section \ref{sect:generalconstruction}.

\ 

(Here (e0) is an axiom schema with infinitely many instances expressing validity of the schema (Pr8) in $\A$.)
\medskip

More precisely: For any nonstandard $\B\models \Th{\N}$
, we construct an exponential $e: B\times A \rightarrow B$ with $\stru{\B}{e}\models Exp$ such that there is an unbounded (in $\B$) set $E\sbs A$ of exponents and (in every coordinate) unbounded set $T\sbs A^3$ of pairwise linearly independent triples such that for every $n\in E$ and $\tuple{a,b,c}\in T$ it is
$$
e(a,n)+e(b,n)=e(c,n).
$$

Moreover, we ensure that $A$ is closed under $e$. 
Hence $\stru{\A}{e}\models \PrA$ $+$ 
\uvz{all open formulas true in $\stru{\B}{e}$} $+$ \uvz{$e$ is total}, and Fermat's Last Theorem for $e$ is 
violated in $\stru{\A}{e}$ by cofinally many exponents $n$ and pairwise linearly independent triples of $a,b,c$.

\medskip

To first outline the idea, take $\mathcal B\models \ISone$. 
To specify the substructure $\mathcal A$, we just need to choose a 
set of \uvz{zeroes} $O$ as in Section \ref{construction}. Our zeroes will be a suitable subset of $\{Q; n|Q$ for all $0<n\in\mathbb N\}$ 
(see Section \ref{sect:FLTviolconstruction})
and $A$ will then consist of elements of the form $Q+z$ for some $Q\in O$ and $z\in\mathbb Z$. Such an $\mathcal A$ will then be a model of
Presburger arithmetic (Lemma \ref{pres}).

We then define the matrices $\varepsilon(Q)$ for $Q\in O$ in such a way that $e(2, Q)=e(3, Q)=e(5, Q)$ for $Q\in O$.
Then we get $e(2, Q+1)+e(3, Q+1)=2e(2, Q)+3e(3, Q)=5e(5, Q)=e(5, Q+1)$, and so $(2,3,5)$ is a counterexample to Fermat's Last Theorem.

This works for any model $\mathcal B$ of $\ISone$ with essentially the same proofs as in the rest of this Section. To obtain the result for
an unbounded set of triples $(a, b, c)$, we shall assume that $\mathcal B$ is a model of $\Th{\N}$, so that we can use the following
number-theoretic result, due to Balog \cite{Bal92}:

\begin{lm}
\label{balog}
For each $K\in\mathbb N$, the equation $3p+5q=2r$ has a solution in primes $p, q, r\in\mathbb N$ such that $p, q, r\geq K$.
\end{lm}

\begin{pf}
The lemma follows by an application of the main theorem of \cite[p. 369]{Bal92}. The matrix $(3, 5, -2)$ is admissible in Balog's sense and satisfies the local solvability conditions (for (C1) choose $3\cdot 1+5\cdot 1-2\cdot 4=0$, for (C2) we can choose $3\cdot (-1)+5\cdot 1-2\cdot 1=0$, which works for any prime power).
Thus the theorem applies in this situation and we know that for sufficiently large $X$, the number of prime solutions with $p, q, r<X$ is at least $\frac{X^2}{(log X)^3}$.

Now fix $K$. We can assume for contradiction that in each prime solution of $3p+5q-2r=0$, at least one of the variables is $<K$. Choose $X$ sufficiently large and let's count the solutions with $p, q, r<X$.
For solutions with $p<K$ we have at most $K$ possibilities for $p$ and $X$ possibilities for $q$. $r$ is then uniquely determined, and so there are at most $KX$ of these solutions. Similarly we have at most $KX$ solutions with $q<K$ and with $r<K$. Hence there are at most $3KX<\frac{X^2}{(log X)^3}$ solutions of the equation, which is a contradiction. 
\end{pf}

\begin{constr}
\label{sect:FLTviolconstruction}
{\bf Construction}
\end{constr}

Let $\B\models \Th{\N}$ be nonstandard. Fix a nonstandard number $\Delta$ from $\B$ and denote by $\bprime$ the $2^\Delta$-th prime of $\B$. 

By Lemma \ref{balog} there is an unbounded (in every coordinate) set $S\sbs B^3$ of pairwise disjoint triples of primes $p,q,r$ such that $3p+5q=2r$
and $p, q, r>P$. 


We may assume that $S$ is definable in $\B$ (e.g. we take the lexicographic order of $B^3$ and define $S$ recursively by 
adding, in each step, the least solution $p,q,r$ disjoint with all previously added.)

We define $\A$ to be the substructure of $\B$ with the universe $A=\set{Q+z}{z\in\Z$ and $n|Q$, $2^{n\Delta}|Q$ for all $0<n\in\N}$ 
(i.e., $A$ is the union of $\sim$-factors of $\B$ which contain an element divisible by all $n>0$ and $2^{n\Delta}$ with $0<n\in\N$).

It is fairly straightforward to check that $\A$ is a model of Presburger arithmetic.

\begin{lm}\label{pres}
$\A\models \PrA$.
\end{lm}

\begin{pf}
One directly checks the axioms (Pr1) -- (Pr8). The only not entirely trivial one is (Pr8):

Take $x\in A$ and $0<n\in\mathbb N$. Then by the construction we have $x=Q+z$ with $n|Q$ and $z\in\mathbb Z$. Thus $Q=na$ and it's easy to see
that $a\in A$ as well. If we now write $z=nb+c$, $0\leq c<n$,
we get $x=n(a+b)+c$ as needed. 
(It is exactly for this argument to work that we require each \uvz{zero} $Q$ to be divisible by all $0<n\in\mathbb N$.)
\end{pf}

Let us also note that $\mathcal A$ contains unboundedly many primes: Fix an element $Q$ such that $n|Q$ and $2^{n\Delta}|Q$ for all $0<n\in\N$. By Dirichlet's Theorem on primes in arithmetic progressions (which holds not only in $\Th{\N}$, but even for any model of PA thanks to an elementary proof by Selberg), $\mathcal B$ contains unboundedly many primes of the form $aQ+1$. Each of these primes in fact lies in $\mathcal A$ by definition. Note that if we assume Dickson's conjecture (which is of course quite strong and far from being proved), $\A$ contains an unbounded set of twin primes of the form $aQ-1$, $aQ+1$.

However, the model $\mathcal A$ is very weak in terms of induction -- it is not even a model of IOpen. Indeed, let $Q$ be again an element such that $n|Q$ and $2^{n\Delta}|Q$ for all $0<n\in\N$, and let $b\in B$, $b\notin A$. Then $Q,Qb\in A$, but the induction axiom for the open formula $Qx\leq Qb$ does not hold in $\A$. (See also Open Problem \ref{open:totality}.)

\

Further, unless stated otherwise, we work in $\B$.

Let us now construct the homomorphism $\vepe$ as in Section \ref{construction}.
For $y\in A$, we set $O_y$ to be the unique element in $[y]_{\sim}$ divisible by all $n$ and $2^{n\Delta}$ with $0<n\in\N$. Clearly, the set $O=\set{O_y}{y\in A}$ is a $\lang{0,+,\dotminus,\cdot}$-substructure of $\A$. (Moreover, $O$ is closed under multiplication by any element $b\in B$.)

For $Q\in O$ we define $\vep{Q} = (\veps{Q}{p}{q})$ as


\begin{tabular}{rcll}
$\veps{Q}{p}{q}$ & $=$ & $Q/2^\Delta$ & for $p,q\leq \bprime$, \\
$\veps{Q}{p}{q}$ & $=$ & $Q/3$ & if $p,q$ are members of the same triple $s\in S$ \\
& & & (allowing $p=q$ lying in some triple in $S$), \\
$\veps{Q}{p}{q}$ & $=$ & $Q$ & for $p=q > \bprime$ and $p$ in no triple $s\in S$, \\
$\veps{Q}{p}{q}$ & $=$ & $0$ & otherwise.
\end{tabular}

\medskip

\begin{lm}
$\vepe : Q \mapsto \vep{Q}$ is a $\lang{0,+,\dotminus,\cdot}$-homomorphism (even an embedding) from $\mathcal{O}$ to $\Pmat$.
\end{lm}

\begin{pf}
To check that $\vepe$ is a homomorphism is easy. All the computations are similar, so as an example, let us just check that the matrices 
$\vep {QR}$ and $\vep {Q}\vep {R}$ have 
the same entries at $(p, q)$ with $(p, q, r)\in S$ for some $r$. We have
$(\vep {Q}\vep {R})_{pq}=\sum_j \veps Qpj\veps Rjq$. Since $\veps Qpj\neq 0$ only for $j=p, q,$ or $r$, we have 
$(\vep {Q}\vep {R})_{pq}=\veps Qpp\veps Rpq + \veps Qpq\veps Rqq + \veps Qpr\veps Rrq=3\cdot (Q/3)\cdot (R/3)=QR/3=\veps {QR}pq$.

We also need to check that $\vep{Q}\in\Pmat$ for every $Q$. Thanks to the definability of the set $S$ in $\B$, $\vep{Q}$ is even definable in $\B$ 
(and obviously there is a definable function $f(q)$ such that all non-zero elements of the $q$-th column are in the rows $p$ with $p\leq f(q)$). 
This is clearly enough since $\B\models\Th{\N}$ codes all finite parts of definable sets.
\end{pf}

By Lemma \ref{lm:vep} we get a semiring homomorphism $\vepe:A\rightarrow\Pmat$ and by definition \refer{def:e} and Proposition \ref{pro:exp} we obtain an exponential $e:B\times A \rightarrow B$ which satisfies the axioms $Exp$.

Let us note that for fixed $y$ the exponential $e(x,y)$ is a definable function of $x$ in $\B$ (this follows from Proposition \ref{pro:exp} and from the definability of $S$).
Moreover, using the new predicate $\mathcal{N}(x)$ expressing \uvz{$x$ is a standard number}, both the set $A$ and the function $y\mapsto O_y$ are definable. Hence, by Proposition \ref{pro:exp} again, $e$ is definable in $\stru{\B}{\mathcal{N}}$.


Now we can show that $e$ is a total exponential on $A$, i.e.,

\begin{lm}
$\Restr{e}{A\times A}:A\times A \rightarrow A$.
\end{lm}

\begin{pf}
Let $x,y\in A$, we want to prove $e(x,y)\in A$. Write $y=Q+\delta$ with $Q\in O$ and $\delta\in\mathbb Z$. It suffices to show that
$e(x, Q)\in A$, for then $e(x, y)=e(x, Q)\cdot x^\delta$ lies also in $A$. Also, we may further suppose that $Q\neq 0$.
Let us now distinguish two cases:

\noindent a) $x$ is divisible by some $p\leq\bprime$: \\
Then $e(x,Q)=e(p,Q)\cdot \alpha$ for some $\alpha\in B$ and $e(p,Q)=\prod_{\bprime\geq q\in \Primes} q^{Q/2^\Delta}$ and for $0<n\in\N$ clearly both 
$n$ and $2^{n\Delta}$ divide $e(p,Q)$. Therefore $e(x,Q)\in O \sbs A$.

\medskip

\noindent b) $x$ is not divisible by any $p\leq\bprime$: \\
By the definition of $\vep Q$ and $e$, in this case also $e(x, Q)$ will not be divisible by any $p\leq P$. 
Since all the entries of $\vep Q$ at positions $(p, q)$ with $p, q>P$
are divisible by $Q/3$, we see that $e(x, Q)=\alpha^{Q/3}$ with $\alpha\in B$ not divisible by any $p\leq P$.

Now note that $\phi(m)$ divides $Q/k$ (here $\phi$ is Euler's totient function) for every $0<m\in\N$ or $m=2^{n\Delta}$. Since $\alpha$ and $m$ are co-prime, we get
$\alpha^{Q/3} \equiv 1 \mod m$ and hence $\alpha^{Q/3}-1\in O$.
Thus $e(x, Q)=\alpha^{Q/3}\in A$.
\end{pf}

From now on denote $\Restr{e}{A\times A}$ just by $e$.

\begin{rmk}
Before discussing Fermat's Last Theorem, observe that various usual elementary number-theoretic statements are not valid with the new exponential $e$, for example Fermat's Little Theorem: Fix $Q\in O$, choose a prime $p=aQ-1>P$ and consider $e(2, p-1)$. By the definition of $e$ we have 
$$4e(2, p-1)=e(2, p+1)=e(2, aQ)=N^{aQ/2^\Delta},$$
where $N=\prod_{q\leq P} q$ is the product of all primes $q$ smaller than our fixed non-standard prime $P$.
Hence 
$$(4e(2, p-1))^{2^\Delta}=N^{aQ}=N^{p+1}\equiv N^2 \pmod p$$ by usual Fermat's Little Theorem in $\mathcal B$. If Fermat's Little Theorem held for $e$, we would have $e(2, p-1)\equiv 1\pmod p$, and so
$$4^{2^\Delta}\equiv N^2\pmod p,$$ i.e., $p| 4^{2^\Delta} - N^2$. There are only finitely many (in the sense of $\mathcal B$) such primes $p$, but infinitely many primes in the arithmetic progression $aQ-1$, a contradiction.
\end{rmk}

Let us now finish the construction of our counterexamples to Fermat's Last Theorem.

\begin{lm}
For every $Q, R\in O$ and every triple $\tuple{p,q,r}\in S$ we have
$$
e(R\cdot 3p, Q + 1) + e(R \cdot 5q, Q + 1) = e(R \cdot 2r, Q + 1).
$$
\end{lm}

\begin{pf}
Note that $e(2, Q)=e(3, Q)=e(5, Q)$ and $e(p, Q)=e(q, Q)=e(r, Q)=(pqr)^{Q/3}$. Thus
we have $e(R\cdot 3p, Q) = e(R\cdot 5q, Q) = e(R\cdot 2r, Q) = e(R,Q)\cdot e(2, Q)\cdot(pqr)^{Q/3}=:K$. 
Then $e(R\cdot 3p, Q+1)=3pKR$, $e(R\cdot 5q, Q+1)=5qKR$, and $e(R\cdot 2r, Q+1)=2rKR$ and the Lemma follows from
$3p+5q=2r$.
\end{pf}

Let us note that while $p,q,r$ may not be in $A$, $R \cdot 3p$, $R \cdot 5q$, $R \cdot 2r$ certainly are in $A$.

\medskip

We summarize our observations as the following:

\begin{thm}\ 
\label{thm:FLT}
\begin{enumerate}[1)]
\item There is a model $\stru{\B}{e}\models \Th{\N} + Exp$ containing an unbounded set $E\sbs B$ of exponents and (in every coordinate) unbounded set $T\sbs B^3$ of pairwise linearly independent triples $\tuple{a,b,c}$ such that for every $n\in E$ and $\tuple{a,b,c}\in T$ we have
$$
e(a,n)+e(b,n)=e(c,n).
$$
Moreover:
\bitem
\item For any fixed $y$, $e(x,y)$ is a definable function of $x$ in $\B$.
\item $e$ is definable in the expansion $\stru{\B}{\mathcal{N}}$ of $\B$ by a predicate $\mathcal{N}(x)$ expressing \uvz{$x$ is a standard number}.
\eitem
\item There is a substructure $\stru{\A}{e}\sbs\stru{\B}{e}$ with $e$ total and $\A\models\PrA$ such that $E\sbs A$, $T\sbs A^3$. (Thus, in addition to axioms of $\PrA$, $\stru{\A}{e}$ satisfies all quantifier-free statements true in $\stru{\B}{e}$.)
\end{enumerate}
\end{thm}

To construct $e$, we used the method described in Section \ref{construction}. Then, necessarily, by Remark \ref{Onotalways}, $A\neq B$, i.e., $e$ is not total on $\B$. 
In general, it is possible to construct a total $e$ by producing a homomorphism $\vepe: B\rightarrow \Pmat$ in a way different from the method of Section \ref{construction} (e.g., see Example \ref{ex:exp}). However, ensuring that Fermat's Last Theorem for $e$ does not hold in the resulting expansion $\stru{\B}{e}$, seems to be a harder question.

\begin{open}
\label{open:totality}
For which arithmetical theories $S$ does there exist a model $\stru{\B}{e}\models S + Exp$ $+$ \uvz{$e$ is total} such that Fermat's Last Theorem for $e$ does not hold in $\stru{\B}{e}$?
In particular, is there such a model for $S = \Th{\N}$?
\end{open}

Let us note that the problem above makes sense only for sufficiently strong theories $S$, since two of the axioms from $Exp$ ((e5) and (e7)) use coding in their formulations. However, if we remove (e7) and replace (e5) with its finite version (e5'), then part 2) of Theorem \ref{thm:FLT} gives the positive answer for $S=\PrA\ +$ \uvz{all open formulas true in the standard model $\N$}.

\section{Catalan's Conjecture}
\label{sect:Catalan}

We show that, unlike Fermat's Last Theorem, Catalan's Conjecture for $e$ (\uvz{the only solution of $e(a,n)-e(b,m)=1$ with $a,b,m,n> 1$ is $a=m=3$, $b=n=2$}) is provable in $\Th{\N} + Exp$. It follows that $\stru{\B}{e}$ and $\stru{\A}{e}$ from Theorem \ref{thm:FLT} are examples of models where FLT for $e$ does not hold but Catalan's Conjecture for $e$ does.

In fact, we can show something slightly stronger, as we need only the axioms (e0) -- (e4) for the exponential function (we denote this set of axioms $Exp^{\prime}$) and we can allow weaker theories than $\Th{\N}$. We mainly need that ABC and Catalan's Conjectures (for the \uvz{original}, definable exponential) hold in our theory.

\bigskip

To briefly review the statement of the ABC Conjecture, let $\mathcal B$ be a model of $\ISone$. Then every element $a$ of $\mathcal B$ has a unique prime factorization
and we can define its radical $\rad a$ as the product of all primes dividing $a$ (discounting multiplicities, i.e., $\rad{24}=6$).
One of the formulations of the ABC Conjecture is:

\begin{conj}[ABC Conjecture]
For every $\vepe>0$ there is $K_\vepe$ such that for all coprime $a, b, c$ with $a+b=c$ we have $c<K_\vepe\rad{abc}^{1+\vepe}$.
\end{conj}

Let us note that Mochizuki has recently announced a proof in the standard model \cite{Moc12}.

\medskip

In the rest of this Section, let $S$ be a theory (in the language of arithmetic $\lang{0,1,+,\cdot,\leq}$) stronger than $\ISone$ such that, for some $K\in\N$, $S$ proves $(\uvz{a,b,c\,$ coprime}$\land a+b=c) \limp c<K\rad{abc}^{1+1/3}$, and Catalan's conjecture (using the exponential $x^y$ definable in $S$). 
By Mochizuki's and Mih\u{a}ilescu's results, we may take $S=\Th{\N}$. 
(We may also conjecture that PA satisfies the property above and take $S=$PA.)

We prove the following:

\begin{thm}
\label{thm:Catalan}
Let $S$ be as above. Catalan's Conjecture for $e$ is provable in $S + Exp^{\prime}$.
\end{thm}

Let $\stru{\B}{e}$ be an arbitrary model of $S+Exp^{\prime}$ and $\A\models\PrA$ be a substructure of $\B$ such that $e: B\times A \rightarrow B$. 
Since we are working with the  weaker set of axioms $Exp^{\prime}$, the exponential need not be given using Proposition \ref{pro:exp} nor the construction from Section {\ref{construction}}. However, we still have the following Lemma.

\begin{lm}\label{rad}
Let $1<x,y\in A$. 
\begin{enumerate}[a)]
\item If $y$ is standard, then  $\rad{e(x, y)}^2 \leq e(x, y)$. 
\item If $y$ is non-standard, then $K\rad{e(x, y)}^n<e(x, y)$ for all standard $K, n$.
\end{enumerate}
\end{lm}

\begin{pf}
a) If $y$ is standard, then $e(x, y)=x^y$ by (e3) and (e4). Hence $\rad{e(x, y)}^2\leq x^2\leq x^y$.

b) Assume that $y$ is non-standard and fix standard $K, n$. Since $\mathcal A\models \PrA$ and $y$ is non-standard, we can 
write $y=n+a$ with $a\in A$ non-standard, and then $a=(n+1)b+m$ with $b\in A$, $0\leq m\leq n$ (by (Pr8)).

We then have $e(x, y)=e(x, n+m+(n+1)b)=x^{n+m}e(x, b)^{n+1}$, and so $\rad{e(x, y)}\leq \rad{x}\rad{e(x, b)}\leq x e(x, b)$.
Thus $K\rad{e(x, y)}^n\leq K x^n e(x, b)^n\leq x^n e(x, b)^{n+1}\leq  x^{n+m}e(x, b)^{n+1}=e(x, y)$ 
(we have used that $K\leq e(x, b)$ for $x > 1$ and $b$ non-standard, which follows from (e3) and (e4)).
\end{pf}

\begin{pro}
\label{pro:catalan}
Catalan's conjecture for $e$ holds in $\stru{\B}{e}$.
\end{pro}

\begin{pf}
Assume that $e(x, a)-e(y, b)=1$, where $x, y, a, b > 1$. We distinguish several cases according to $a,b$.

1) If $a, b$ are both standard then this is just Catalan's conjecture in $\mathcal B$.

2) Assume $a$ is non-standard. By the ABC conjecture for $\vepe=1/3$ (which is provable in $S$) we have $e(x, a)<K\rad{e(x,a)e(y,b)}^{1+\vepe}$, and so using Lemma \ref{rad} we have (note that $3+3\vepe=4$)

$$e(x, a)e(y, b)^2<e(x, a)^3<K^3\rad{e(x,a)e(y,b)}^{3+3\vepe}\leq$$
  $$\leq (K^3\rad{e(x,a)}^{4}) \rad{e(y,b)}^{4}
  < e(x, a) e(y, b)^2,$$

a contradiction.  

Let us note that to show $\rad{e(y,b)}^4<e(y,b)^2$ (in the last inequality), we use Lemma \ref{rad} a) if $b$ is standard, or b) otherwise.

3) The case of $b$ non-standard is analogous -- we only need to start from $e(x, a)^2e(y, b)$ instead of $e(x, a)e(y, b)^2$.
\end{pf}

This proves Theorem \ref{thm:Catalan}.

\section{Coprimality}\label{sect:coprime}

One may naturally wonder what causes the difference between the validity of Fermat's Last Theorem and Catalan's Conjecture and how is it possible that Catalan's Conjecture holds even with our weak exponential. It appears to us that the main number-theoretic weakness of the exponential is the fact that
for coprime $x$ and $y$, the values $e(x, a)$ and $e(y, b)$ need not be coprime. In fact, we have crucially exploited this in the construction of our counterexamples to Fermat's Last Theorem in Section \ref{sect:FLTviolconstruction}.
However, this does not play any role when considering Catalan's Conjecture, as $e(x, a)-e(y, b)=1$ immediately forces $e(x, a)$ and $e(y, b)$ to be coprime (and then we can apply ABC Conjecture).

Let us thus consider the following additional axiom for the exponential $e$:

\medskip

(e8) \uvz{If $x$ and $y$ are coprime, then so are $e(x, a)$ and $e(y, b)$.}

\medskip

This is equivalent to all corresponding matrices $\varepsilon(a)$ being diagonal -- but as Example \ref{ex:exp} shows, the exponential can still be different from the usual one. 

In fact, such \uvz{diagonal} homomorphisms $\vepe:A\rightarrow\Pmat$ are exactly homomorphisms of the form $\vep{a}=diag(f_p(a);p\in\Prv)$, where $f_p:A\rightarrow B$ are homomorphisms and $diag$ denotes the diagonal matrix. Hence there are $|\mathrm{Hom}(\A,\B)|^\omega$ exponentials satisfying $Exp +$ (e8), namely those given as $e_f(\prod_i p_i^{e_i},a)=\prod_i p_i^{e_if_{p_i}(a)}$ with $f=(f_p;p\in\Prv)$ homomorphisms from $\A$ to $\B$.
(If $A$ is closed under subtraction, then by Remark \ref{rem:zerosmatrices} all homomorphisms $f_p$ are necessarily injective.)

Note that (e8) is still much weaker than induction for $e$. Indeed, only the usual exponential $x^y$ satisfies induction: if some exponential $e$ satisfied induction ($\Sigma_1$-induction would be enough), it would be total, and we could use the induction to prove that $e(x,y)=x^y$ for all $y$.

\

Then we have a direct analogue of Theorem \ref{thm:Catalan}: Let $T$ be a theory (in the language of arithmetic $\lang{0,1,+,\cdot,\leq}$) stronger than $\ISone$ such that, for some $K\in\N$ and some $\varepsilon>0$, $T$ proves $(\uvz{a,b,c\,$ coprime}$\land a+b=c) \limp c<K\rad{abc}^{1+\varepsilon}$, and Fermat's Last Theorem (using the exponential $x^y$ definable in $T$). 
We may again take $T=\Th{\N}$.

\begin{thm}\label{thm 6.1}
Let $T$ be a theory as above. Fermat's Last Theorem for $e$ is provable in $T + Exp^\prime +$ (e$5^\prime$) + (e8).
\end{thm}

Let us recall that $Exp^\prime$ denotes the axioms (e0) -- (e4).

\begin{pf}
The proof is analogous to that of Theorem \ref{thm:Catalan}, in fact it is a little easier:

Assume that $e(x, n)+e(y, n)=e(z, n)$.
First we use (e$5^\prime$) to divide the equation by $e(g, n)$, where $g$ is the greatest common divisor of $x, y, z$.
Thus we can restrict ourselves to the situation with $x, y, z$ coprime. Hence also $e(x, n),e(y, n)$ and $e(z, n)$ are mutually coprime by (e8).

By the usual Fermat's Last Theorem we can also assume that $n$ is non-standard.
By the ABC Conjecture and Lemma \ref{rad} b) we have 
$$e(z, n)<K\rad{e(x, n)e(y, n)e(z, n)}^{1+\varepsilon}\leq K \left[\rad{e(x, n)}\rad{e(y, n)}\rad{e(z, n)} \right]^{1+\varepsilon}\leq$$
$$\leq (e(x, n)e(y, n)e(z, n))^{1/3}<e(z, n),$$
a contradiction.
\end{pf}

Theorem \ref{thm 6.1} seems to suggest that (at the very least in the class of models we are considering) full mathematical induction for the exponential function is not necessary to prove Fermat's Last Theorem (and Catalan's conjecture), but rather that it suffices to have one particular consequence of it, namely the coprimality property (e8).

\

We find very interesting the question whether there is a model $\stru{\B}{e}\models\Th{\N}+Exp$ with $e$ total where FLT for $e$ does not hold (see Open Problem \ref{open:totality}). In the light of Theorem \ref{thm 6.1} we now see that such $e$ would have to be given by a \uvz{non-diagonal} $\vepe:B\rightarrow\Pmat$. We therefore state the following

\begin{open}
Is there a model $\B\models\Th{\N}$ (or at least of $\ISone$) that permits a semiring homomorphism $\vepe:B\rightarrow\Pmat$ with some values $\vep{b}$ non-diagonal?
\end{open}

\bibliography{biblio}{}
\bibliographystyle{amsalpha}

\end{document}